\documentclass{amsart}

\usepackage{amsfonts}
\usepackage{amsmath}
\usepackage{amssymb}
\usepackage{latexsym}
\usepackage{amscd}
\usepackage{amsthm}
\usepackage{subfigure}
\usepackage{graphicx}
\usepackage[all]{xy}
\usepackage{color}

\newtheorem{thm}{Theorem}[section]
\newtheorem{cor}[thm]{Corollary}
\newtheorem{prob}[thm]{Problem}
\theoremstyle{definition}

\newtheorem{rem}[thm]{Remark}

\newcommand{\M}{\mathcal{M}}
\newcommand{\T}{\mathcal{T}}
\newcommand{\Y}{\mathcal{Y}}
\newcommand{\U}{\mathcal{U}}

\author[R. Kobayashi]{Ryoma Kobayashi}
\address{
Department of General Education,\endgraf
National Institute of Technology, Ishikawa College,\endgraf
Tsubata, Ishikawa, 929-0392, Japan
}
\email{kobayashi\_ryoma@ishikawa-nct.ac.jp}



\thanks{The author was supported by JSPS KAKENHI Grant Number JP19K14542 and 22K13920.}

\begin{document}

\title[Simple infinite presentations for $\M(N_{g,n})$]
{Simple infinite presentations for the mapping class group of a compact non-orientable surface}

\maketitle

\begin{abstract}
Omori and the author~\cite{KO1} have given an infinite presentation for the mapping class group of a compact non-orientable surface.
In this paper, we give more simple infinite presentations for this group.
\end{abstract}

\section{Introduction}

For $g\geq1$ and $n\geq0$, we denote by $N_{g,n}$ a surface obtained by removing $n$ disjoint open disks from a connected sum of $g$ real projective planes, and call this surface a compact non-orientable surface of genus $g$ with $n$ boundary components.
We can regard $N_{g,n}$ as a surface obtained by attaching $g$ M\"obius bands to $g$ boundary components of a sphere with $g+n$ boundary components, as shown in Figure~\ref{non-ori-surf}.
We call these attached M\"obius bands \textit{crosscaps}.

\begin{figure}[htbp]
\includegraphics{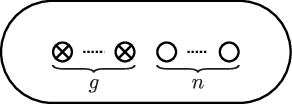}
\caption{A model of a non-orientable surface $N_{g,n}$.}\label{non-ori-surf}
\end{figure}

The \textit{mapping class group} $\M(N_{g,n})$ of $N_{g,n}$ is defined as the group consisting of isotopy classes of all diffeomorphisms of $N_{g,n}$ that fix the boundary pointwise.
$\M(N_{1,0})$ and $\M(N_{1,1})$ are trivial (see \cite{E}).
Finite presentations for $\M(N_{2,0})$, $\M(N_{2,1})$, $\M(N_{3,0})$ and $\M(N_{4,0})$ ware given by \cite{Li1}, \cite{St1}, \cite{BC} and \cite{Sz} respectively.
Paris-Szepietowski~\cite{PS} gave a finite presentation of $\M(N_{g,n})$ with \textit{Dehn twists} and \textit{crosscap transpositions} for $g+n>3$ with $n\leq1$.
Stukow~\cite{St2} gave another finite presentation of $\M(N_{g,n})$ with Dehn twists and one \textit{crosscap slide} for $g+n>3$ with $n\leq1$, applying Tietze transformations for the presentation of $\M(N_{g,n})$ given in \cite{PS}.
Omori~\cite{O} gave an infinite presentation of $\M(N_{g,n})$ with all Dehn twists and all crosscap slides for $g\geq1$ and $n\leq1$, using the presentation of $\M(N_{g,n})$ given in \cite{St2}, and then, following this work, Omori and the author~\cite{KO1} gave an infinite presentation of $\M(N_{g,n})$ with all Dehn twists and all crosscap slides for $g\geq1$ and $n\geq2$.
In this paper, we give four more simple infinite presentations of $\M(N_{g,n})$ with all Dehn twists and all crosscap slides, and with all Dehn twists and all crosscap transpositions.

Through this paper, the product $gf$ of mapping classes $f$ and $g$ means that we apply $f$ first and then $g$.
Moreover we do not distinguish a simple closed curve from its isotopy class.

We define a \textit{Dehn twist}, a \textit{crosscap slide} and a \textit{crosscap transposition} which are elements of $\M(N_{g,n})$.
For a simple closed curve $c$ of $N_{g,n}$, a regular neighborhood of $c$ is either an annulus or a M\"obius band.
We call $c$ a \textit{two sided} or a \textit{one sided} simple closed curve respectively.
For a two sided simple closed curve $c$, we can take two orientations $+_c$ and $-_c$ of a regular neighborhood of $c$.
The \textit{right handed Dehn twist} $t_{c;\theta}$ about a two sided simple closed curve $c$ with respect to $\theta\in\{+_c,-_c\}$ is the isotopy class of the map as shown in Figure~\ref{dehn-slide-trans}~(a).
Note that for an oriented two sided simple closed curve $\alpha$ and $\theta\in\{+_\alpha,-_\alpha\}$, we regard $t_{\alpha,\theta}=t_{\alpha^{-1},\theta}$, where $\alpha^{-1}$ is the inverse loop of $\alpha$.
We write $t_{c;\theta}=t_c$ if the orientation $\theta$ is given explicitly, that is, the direction of the twist is indicated by an arrow written beside $c$ as shown in Figure~\ref{dehn-slide-trans}~(a).
For a one sided simple closed curve $\mu$ of $N_{g,n}$ and an oriented two sided simple closed curve $\alpha$ of $N_{g,n}$ such that $N_{g,n}\setminus\{\alpha\}$ is non-orientable when $g\geq3$ and that $\mu$ and $\alpha$ intersect transversely at only one point, the \textit{crosscap slide} $Y_{\mu,\alpha}$ about $\mu$ and $\alpha$ is the isotopy class of the map described by pushing the crosscap which is a regular neighborhood of $\mu$ once along $\alpha$, as shown in Figure~\ref{dehn-slide-trans}~(b).
The \textit{crosscap transposition} $U_{\mu,\alpha;\theta}$ about $\mu$ and $\alpha$ with respect to $\theta\in\{+_\alpha,-_\alpha\}$ is defined as $U_{\mu,\alpha;\theta}=t_{\alpha;\theta}Y_{\mu,\alpha}$, as shown in Figure~\ref{dehn-slide-trans}~(b).
Note that $Y_{\mu,\alpha}$ and $U_{\mu,\alpha;\theta}$ can not be defined when $g=1$.
We have $Y_{\mu,\alpha}(\alpha)=\alpha$ and $(Y_{\mu,\alpha})_\ast(\pm_\alpha)=\mp_\alpha$, where $f_\ast(\theta)$ is the orientation of a regular neighborhood of $f(c)$ induced from $\theta\in\{+_c,-_c\}$, for a two sided simple closed curve $c$ of $N_{g,n}$ and $f\in\M(N_{g,n})$.
Since $t_{\alpha;\theta}(\alpha)=\alpha$ and $(t_{\alpha;\theta})_\ast(\pm_\alpha)=\pm_\alpha$, we also have $U_{\mu,\alpha;\theta}(\alpha)=\alpha$ and $(U_{\mu,\alpha;\theta})_\ast(\pm_\alpha)=\mp_\alpha$.

\begin{figure}[htbp]
\subfigure[The Dehn twist $t_{c;\theta}=t_{c}$ about $c$ with respect to $\theta\in\{+_\alpha,-_\alpha\}$.]{\includegraphics{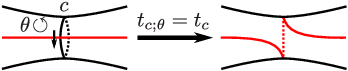}}

\subfigure[The crosscap slide $Y_{\mu,\alpha}$ about $\mu$ and $\alpha$, and the crosscap transposition $U_{\mu,\alpha:\theta}$ about $\mu$ and $\alpha$ with respect to $\theta\in\{+_\alpha,-_\alpha\}$.
The action by $U_{\mu,\alpha:\theta}$ causes the two crosscaps to swap.]{\includegraphics{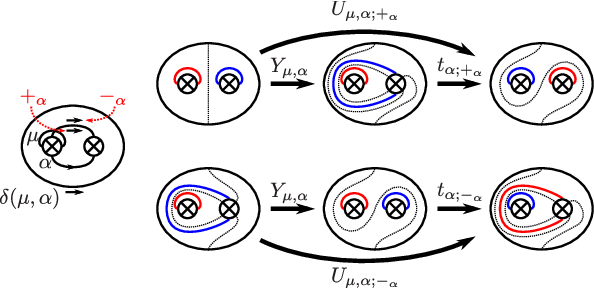}}
\caption{}\label{dehn-slide-trans}
\end{figure}

Let $c_1,\dots,c_k$, $c_0$, $c_0^\prime$ and $d_1,\dots,d_7$ be simple closed curves with arrows of a surface as shown in Figure~\ref{chain-lantern}.
$\M(N_{g,n})$ admits following relations.
\begin{itemize}
\item	\begin{itemize}
	\item	$(t_{c_1}t_{c_2}\cdots{}t_{c_k})^{k+1}=t_{c_0}t_{c_0^\prime}$ if $k$ is odd,\\
	\item	$(t_{c_1}t_{c_2}\cdots{}t_{c_k})^{2k+2}=t_{c_0}$ if $k$ is even.
	\end{itemize}
\item	$t_{d_1}t_{d_2}t_{d_3}=t_{d_4}t_{d_5}t_{d_6}t_{d_7}$.
\end{itemize}
These relations are called a \textit{$k$-chain relation} and a \textit{lantern relation} respectively.

\begin{figure}[htbp]
\includegraphics{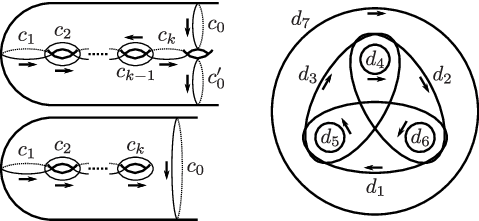}
\caption{Simple closed curves $c_1,\dots,c_k$, $c_0$, $c_0^\prime$ and $d_1,\dots,d_7$ with arrows of a surface.}\label{chain-lantern}
\end{figure}

Let $\T$, $\Y$ and $\U\subset\M(N_{g,n})$ denote the sets consisting of all Dehn twists, all crosscap slides and all crosscap transpositions respectively.
Our main results are as follows.

\begin{thm}\label{main-thm-1}
For $g\geq1$ and $n\geq0$, $\M(N_{g,n})$ admits a presentation with a generating set $\T\cup\Y$.
The defining relations are
\begin{enumerate}
\item	$t_{c;\theta}=1$ if $c$ bounds a disk or a M\"obius band,
\item	\begin{enumerate}
	\item	$t_{c;+_c}^{-1}=t_{c;-_c}$ for any $t_{c;+_c}\in\T$,
	\end{enumerate}
\item	\begin{enumerate}
	\item	$ft_{c;\theta}f^{-1}=t_{f(c);f_\ast(\theta)}$ for any $t_{c;\theta}\in\T$ and $f\in\T\cup\Y$,
	\item	$fY_{\mu,\alpha}f^{-1}=Y_{f(\mu),f(\alpha)}$ for any $Y_{\mu,\alpha}\in\Y$ and $f\in\T$,
	\end{enumerate}
\item	all the $2$-chain relations,
\item	all the lantern relations and
\item	$Y_{\mu,\alpha}^2=t_{\delta(\mu,\alpha)}$ for any $Y_{\mu,\alpha}\in\Y$, where $\delta(\mu,\alpha)$ is a simple closed curve with an arrow determined by $\mu$ and $\alpha$ as shown in Figure~\ref{dehn-slide-trans}~(b).
\end{enumerate}
\end{thm}

\begin{thm}\label{main-thm-2}
For $g\neq2$ and $n\geq0$, $\M(N_{g,n})$ admits a presentation with a generating set $\T\cup\Y$.
The defining relations are
\begin{enumerate}
\item	$t_{c;\theta}=1$ if $c$ bounds a disk or a M\"obius band,
\item	\begin{enumerate}
	\item	$t_{c;+_c}^{-1}=t_{c;-_c}$ for any $t_{c;+_c}\in\T$,
	\item	$Y_{\mu,\alpha}^{-1}=Y_{\mu,\alpha^{-1}}$ for any $Y_{\mu,\alpha}\in\Y$,
	\end{enumerate}
\item	\begin{enumerate}
	\item$ft_{c;\theta}f^{-1}=t_{f(c);f_\ast(\theta)}$ for any $t_{c;\theta}\in\T$ and $f\in\T\cup\Y$,
	\item	$fY_{\mu,\alpha}f^{-1}=Y_{f(\mu),f(\alpha)}$ for any $Y_{\mu,\alpha}\in\Y$ and $f\in\T$,
	\end{enumerate}
\item	all the $2$-chain relations and
\item	all the lantern relations.
\end{enumerate}
\end{thm}

\begin{thm}\label{main-thm-3}
For $g\geq1$ and $n\geq0$, $\M(N_{g,n})$ admits a presentation with a generating set $\T\cup\U$.
The defining relations are
\begin{enumerate}
\item	$t_{c;\theta}=1$ if $c$ bounds a disk or a M\"obius band,
\item	\begin{enumerate}
	\item	$t_{c;+_c}^{-1}=t_{c;-_c}$ for any $t_{c;+_c}\in\T$,
	\end{enumerate}
\item	\begin{enumerate}
	\item	$ft_{c;\theta}f^{-1}=t_{f(c);f_\ast(\theta)}$ for any $t_{c;\theta}\in\T$ and $f\in\T\cup\U$,
	\item	$fU_{\mu,\alpha;\theta}f^{-1}=U_{f(\mu),f(\alpha);f_\ast(\theta)}$ for any $U_{\mu,\alpha;\theta}\in\U$ and $f\in\T\cup\{U_{\mu,\alpha;\theta}\}$,
	\end{enumerate}
\item	all the $2$-chain relations,
\item	all the lantern relations and
\item	$U_{\mu,\alpha;\theta}^2=t_{\delta(\mu,\alpha)}$ for any $U_{\mu,\alpha;\theta}\in\U$, where $\delta(\mu,\alpha)$ is a simple closed curve with an arrow determined by $\mu$ and $\alpha$ as shown in Figure~\ref{dehn-slide-trans}~(b).
\end{enumerate}
\end{thm}

\begin{thm}\label{main-thm-4}
For $g\neq2$ and $n\geq0$, $\M(N_{g,n})$ admits a presentation with a generating set $\T\cup\U$.
The defining relations are
\begin{enumerate}
\item	$t_{c;\theta}=1$ if $c$ bounds a disk or a M\"obius band,
\item	\begin{enumerate}
	\item	$t_{c;+_c}^{-1}=t_{c;-_c}$ for any $t_{c;+_c}\in\T$,
	\item	$U_{\mu,\alpha;\theta}^{-1}=U_{\mu,\alpha^{-1};\theta}$ for any $U_{\mu,\alpha;\theta}\in\U$,
	\end{enumerate}
\item	\begin{enumerate}
	\item	$ft_{c;\theta}f^{-1}=t_{f(c);f_\ast(\theta)}$ for any $t_{c;\theta}\in\T$ and $f\in\T\cup\U$,
	\item	$fU_{\mu,\alpha;\theta}f^{-1}=U_{f(\mu),f(\alpha);f_\ast(\theta)}$ for any $U_{\mu,\alpha;\theta}\in\U$ and $f\in\T\cup\{U_{\mu,\alpha;\theta}\}$,
	\end{enumerate}
\item	all the $2$-chain relations and
\item	all the lantern relations.
\end{enumerate}
\end{thm}

The relation~(6) of Theorem~\ref{main-thm-1} and the relation~(2)~(b) of Theorem~\ref{main-thm-2} are special forms of a relation on crosscap slides and Dehn twists of $\M(N_{g,n})$ given in \cite{O} and \cite{KO1}.
In this sense one could say that we have given more simple infinite presentations for $\M(N_{g,n})$.

In Section~\ref{proof-main-thm-1}, we prove Theorem~\ref{main-thm-1}.
In Section~\ref{proof-main-thm-2}, we prove Theorem~\ref{main-thm-2}.
In Section~\ref{proof-main-thm-3-4}, we prove Theorems~\ref{main-thm-3} and \ref{main-thm-4}.
Finally, in Appendix~\ref{problem}, we introduce some problems on presentations for $\M(N_{g,n})$.

\section{Proof of Theorem~\ref{main-thm-1}}\label{proof-main-thm-1}

We denote by $\T(N_{g,n})$ the subgroup of $\M(N_{g,n})$ generated by $\T$, and call the \textit{twist subgroup} of $\M(N_{g,n})$.
Omori and the author~\cite{KO2} gave an infinite presentation for $\T(N_{g,n})$ as follows.

\begin{thm}[\cite{KO2}]\label{IFP-T}
For $g\geq1$ and $n\geq0$, $\T(N_{g,n})$ admits a presentation with a generating set $\T$.
The defining relations are
\begin{enumerate}
\item	$t_{c;\theta}=1$ if $c$ bounds a disk or a M\"obius band,
\item	$t_{c;+_c}^{-1}=t_{c;-_c}$ for any $t_{c;+_c}\in\T$,
\item	$ft_{c;\theta}f^{-1}=t_{f(c);f_\ast(\theta)}$ for any $t_{c;\theta}$ and $f\in\T$,
\item	all the $2$-chain relations and
\item	all the lantern relations.
\end{enumerate}
\end{thm}

As a corollary of Theorem~\ref{IFP-T}, we obtain the following.

\begin{cor}\label{IFP-M-1}
Fix one crosscap slide $Y_{\mu_0,\alpha_0}\in\Y$.
For $g\geq2$ and $n\geq0$, $\M(N_{g,n})$ admits a presentation with a generating set $\T\cup\{Y_{\mu_0,\alpha_0}\}$.
The defining relations are
\begin{enumerate}
\item	$t_{c;\theta}=1$ if $c$ bounds a disk or a M\"obius band,
\item	$t_{c;+_c}^{-1}=t_{c;-_c}$ for any $t_{c;+_c}\in\T$,
\item	$ft_{c;\theta}f^{-1}=t_{f(c);f_\ast(\theta)}$ for any $t_{c;\theta}\in\T$ and $f\in\T\cup\{Y_{\mu_0,\alpha_0}\}$,
\item	all the $2$-chain relations,
\item	all the lantern relations and
\item	$Y_{\mu_0,\alpha_0}^2=t_{\delta(\mu_0,\alpha_0)}$.
\end{enumerate}
\end{cor}

\begin{proof}
For $g\geq2$, we have the short exact sequence
$$1\to\T(N_{g,n})\to\M(N_{g,n})\to\langle{Y_{\mu_0,\alpha_0}\mid{}Y_{\mu_0,\alpha_0}^2}\rangle\to1$$
(see \cite{Li1,Li2}).
As basics on combinatorial group theory, a presentation for $\M(N_{g,n})$ is obtained by adding the generator $Y_{\mu_0,\alpha_0}$ and the relations $Y_{\mu_0,\alpha_0}t_{c;\theta}Y_{\mu_0,\alpha_0}^{-1}=t_{Y_{\mu_0,\alpha_0}(c);(Y_{\mu_0,\alpha_0})_\ast(\theta)}$ and $Y_{\mu_0,\alpha_0}^2=t_{\delta(\mu_0,\alpha_0)}$ to the presentation for $\T(N_{g,n})$ given in Theorem~\ref{IFP-T} (for details, for instance see~\cite{J}).
Thus we obtained the claim.
\end{proof}

We now prove Theorem~\ref{main-thm-1}.

Let $G$ be the group presented in Theorem~\ref{main-thm-1}.
When $g=1$, since a crosscap slide can not be defined we see $\M(N_{g,n})=\T(N_{g,n})$.
In addition, since $\Y$ is the empty set, $G$ is isomorphic to $\T(N_{g,n})$ by Theorem~\ref{IFP-T}, and so $\M(N_{g,n})$.
Therefore we suppose $g\geq2$.

Let $\varphi:G\to\M(N_{g,n})$ be the natural homomorphism and $\psi:\M(N_{g,n})\to{G}$ the homomorphism defined as $\psi(f)=f$ for any $f\in\T\cup\{Y_{\mu_0,\alpha_0}\}$.
Since any relation of $\M(N_{g,n})$ in Corollary~\ref{IFP-M-1} is satisfied in $G$, we see that $\psi$ is well-defined.
Since $\varphi\circ\psi$ is the identity map clearly, it suffices to show that $\psi\circ\varphi$ is the identity map.

For any $t_{c;\theta}\in\T$, it is clear that $\psi(\varphi(t_{c;\theta}))=t_{c;\theta}$.
For any $Y_{\mu,\alpha}\in\Y$, there is $f\in\M(N_{g,n})$ such that $fY_{\mu_0,\alpha_0}f^{-1}=Y_{\mu,\alpha}$, that is, $f(\mu_0)=\mu$ and $f(\alpha_0)=\alpha$, in $\M(N_{g,n})$.
If $f$ is in $\T(N_{g,n})$, since $f$ can be represented as a word $f_1\cdots{}f_k$ on $\T$, repeating the relations~(2)~(a) and (3)~(b) of $G$, we calculate
\begin{eqnarray*}
\psi(\varphi(Y_{\mu,\alpha}))
&=&
\psi(Y_{\mu,\alpha})\\
&=&
\psi(fY_{\mu_0,\alpha_0}f^{-1})\\
&=&
\psi(f_1\cdots{}f_kY_{\mu_0,\alpha_0}f_k^{-1}\cdots{}f_1^{-1})\\
&=&
f_1\cdots{}f_kY_{\mu_0,\alpha_0}f_k^{-1}\cdots{}f_1^{-1}\\
&=&
Y_{f_1\cdots{}f_k(\mu_0),f_1\cdots{}f_k(\alpha_0)}\\
&=&
Y_{f(\mu_0),f(\alpha_0)}\\
&=&
Y_{\mu,\alpha}.
\end{eqnarray*}
If $f$ is not in $\T(N_{g,n})$, there exists $h\in\T(N_{g,n})$ such that $f=hY_{\mu_0,\alpha_0}$, by the sequence in the proof of Corollary~\ref{IFP-M-1}.
Since $h$ can be represented as a word $h_1\cdots{}h_k$ on $\T$, repeating the relations~(2)~(a) and (3)~(b) of $G$, we calculate
\begin{eqnarray*}
\psi(\varphi(Y_{\mu,\alpha}))
&=&
\psi(Y_{\mu,\alpha})\\
&=&
\psi(fY_{\mu_0,\alpha_0}f^{-1})\\
&=&
\psi(hY_{\mu_0,\alpha_0}Y_{\mu_0,\alpha_0}Y_{\mu_0,\alpha_0}^{-1}h^{-1})\\
&=&
\psi(h_1\cdots{}h_kY_{\mu_0,\alpha_0}h_k^{-1}\cdots{}h_1^{-1})\\
&=&
h_1\cdots{}h_kY_{\mu_0,\alpha_0}h_k^{-1}\cdots{}h_1^{-1}\\
&=&
Y_{h_1\cdots{}h_k(\mu_0),h_1\cdots{}h_k(\alpha_0)}\\
&=&
Y_{h(\mu_0),h(\alpha_0)}\\
&=&
Y_{h(Y_{\mu_0,\alpha_0}(\mu_0)),h(Y_{\mu_0,\alpha_0}(\alpha_0))}\\
&=&
Y_{f(\mu_0),f(\alpha_0)}\\
&=&
Y_{\mu,\alpha}.
\end{eqnarray*}
Therefore we conclude that $\psi\circ\varphi$ is the identity map.

Thus we have that $G$ is isomorphic to $\M(N_{g,n})$, and hence the proof of Theorem~\ref{main-thm-1} is completed.

\section{Proof of Theorem~\ref{main-thm-2}}\label{proof-main-thm-2}

Any relation which appeared in Theorem~\ref{main-thm-2} is clearly satisfied in $\M(N_{g,n})$.
In addition, the relations~(1)-(5) in Theorem~\ref{main-thm-1} are included in the relations in Theorem~\ref{main-thm-2}.
Hence it suffices to show that the relation~(6) in Theorem~\ref{main-thm-1} is obtained from relations in Theorem~\ref{main-thm-2} when $g\geq3$.

\begin{rem}
\begin{itemize}
\item	In the relation~(3)~(a) in the main theorems, if $f=t_{c^\prime;\theta^\prime}$, $|c\cap{}c^\prime|=0$ or $1$, and the orientations $\theta$ and $\theta^\prime$ are compatible, then the relation can be rewritten as a \textit{commutativity relation} $t_{c;\theta}t_{c^\prime;\theta^\prime}=t_{c^\prime;\theta^\prime}t_{c;\theta}$ or a \textit{braid relation} $t_{c;\theta}t_{c^\prime;\theta^\prime}t_{c;\theta}=t_{c^\prime;\theta^\prime}t_{c;\theta}t_{c^\prime;\theta^\prime}$ respectively.
	We assign the label $(3)^\prime$ to all the commutativity relations and all the braid relations.
\item	It is known that any chain relation is obtained from the relations~(1), (3)~(a), (4) and (5) in the main theorems (see \cite{Lu}).
	We assign the label $(4)^\prime$ to all the chain relations.
\end{itemize}
\end{rem}

For any $Y_{\mu,\alpha}\in\Y$, let $\beta$, $\gamma$, $\delta$, $A$, $B$, $C$, $D$ and $E$ be simple closed curves with arrows as shown in Figure~\ref{B2}.
By Figure~\ref{B3}~(a), we have $t_\beta{}t_\alpha{}t_\gamma{}t_\beta{}t_\alpha{}t_\beta(\mu)=\mu$ and $t_\beta{}t_\alpha{}t_\gamma{}t_\beta{}t_\alpha{}t_\beta(\alpha^{-1})=\alpha$.
By Figure~\ref{B3}~(b) and the relation~(2)~(a), we have $t_{Y_{\mu,\alpha^{-1}}(\beta)}^{-1}=t_\delta$ and $t_{Y_{\mu,\alpha^{-1}}(\gamma)}^{-1}=t_\beta$.
Note that $t_{Y_{\mu,\alpha^{-1}}(\alpha)}^{-1}=t_\alpha$.
By Figure~\ref{B3}~(c) and the relation~(2)~(a), we have $t_{t_\beta{}t_\alpha(\gamma)}=t_A$ and $t_{t_\beta{}t_\alpha(\delta)}^{-1}=t_B$.
By Figure~\ref{B3}~(d) and the relation~$(4)^\prime$, we have $(t_\delta{}t_\alpha{}t_\beta)^4=t_Ct_D$.
By Figure~\ref{B3}~(e) and the relations~(2)~(a) and (5), we have $t_B^{-1}t_A^{-1}t_{\delta(\mu,\alpha)}=t_Ct_Et_\alpha{}t_\alpha^{-1}$.

\begin{figure}[htbp]
\includegraphics{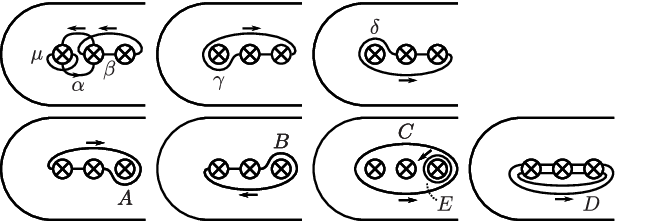}
\caption{Simple closed curves $\mu$, $\alpha$, $\beta$, $\gamma$, $\delta$ $A$, $B$, $C$, $D$ and $E$ of $N_{g,n}$.}\label{B2}
\end{figure}

\begin{figure}[htbp]
\subfigure[$t_\beta{}t_\alpha{}t_\gamma{}t_\beta{}t_\alpha{}t_\beta(\mu)=\mu$ and $t_\beta{}t_\alpha{}t_\gamma{}t_\beta{}t_\alpha{}t_\beta(\alpha^{-1})=\alpha$.]{\includegraphics{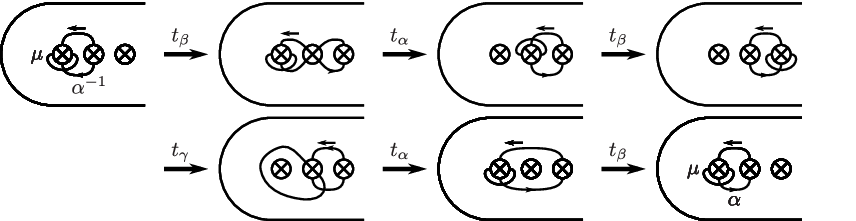}}

\subfigure[$t_{Y_{\mu,\alpha^{-1}}(\beta)}^{-1}=t_\delta$ and $t_{Y_{\mu,\alpha^{-1}}(\gamma)}^{-1}=t_\beta$.]{\includegraphics{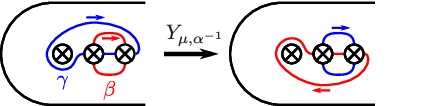}}

\subfigure[$t_{t_\beta{}t_\alpha(\gamma)}=t_A$ and $t_{t_\beta{}t_\alpha(\delta)}^{-1}=t_B$.]{\includegraphics{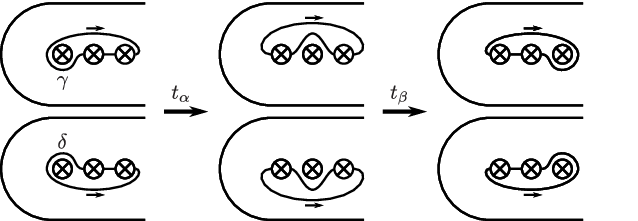}}

\subfigure[$(t_\delta{}t_\alpha{}t_\beta)^4=t_Ct_D$.]{\includegraphics{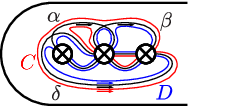}}

\subfigure[$t_B^{-1}t_A^{-1}t_{\delta(\mu,\alpha)}=t_Ct_Et_\alpha{}t_\alpha^{-1}$.]{\includegraphics{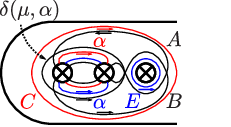}}
\caption{}\label{B3}
\end{figure}

We calculate
\begin{eqnarray*}
Y_{\mu,\alpha}^2
&=&
Y_{t_\beta{}t_\alpha{}t_\gamma{}t_\beta{}t_\alpha{}t_\beta(\mu),t_\beta{}t_\alpha{}t_\gamma{}t_\beta{}t_\alpha{}t_\beta(\alpha^{-1})}Y_{\mu,\alpha}\\
&\overset{\mathrm{(3)(b),(2)(b)}}{=}&
t_\beta{}t_\alpha{}t_\gamma{}t_\beta{}t_\alpha{}t_\beta{}Y_{\mu,\alpha^{-1}}t_\beta^{-1}t_\alpha^{-1}t_\beta^{-1}t_\gamma^{-1}t_\alpha^{-1}t_\beta^{-1}Y_{\mu,\alpha^{-1}}^{-1}\\
&\overset{\mathrm{(3)(a)}}{=}&
t_\beta{}t_\alpha{}t_\gamma{}t_\beta{}t_\alpha{}t_\beta{}t_{Y_{\mu,\alpha^{-1}}(\beta)}^{-1}t_{Y_{\mu,\alpha^{-1}}(\alpha)}^{-1}t_{Y_{\mu,\alpha^{-1}}(\beta)}^{-1}t_{Y_{\mu,\alpha^{-1}}(\gamma)}^{-1}t_{Y_{\mu,\alpha^{-1}}(\alpha)}^{-1}t_{Y_{\mu,\alpha^{-1}}(\beta)}^{-1}\\
&=&
t_\beta{}t_\alpha{}t_\gamma\underline{t_\beta{}t_\alpha{}t_\beta}t_\delta{}t_\alpha\underline{t_\delta{}t_\beta}t_\alpha{}t_\delta\\
&=&
t_\beta{}t_\alpha{}t_\gamma{}t_\alpha{}t_\beta{}t_\alpha{}t_\delta{}t_\alpha{}t_\beta\underline{t_\delta{}t_\alpha{}t_\delta}\\
&=&
t_\beta{}t_\alpha{}t_\gamma{}t_\alpha{}t_\beta{}t_\alpha{}t_\delta\underline{t_\alpha{}t_\beta{}t_\alpha}t_\delta{}t_\alpha\\
&=&
t_\beta{}t_\alpha{}t_\gamma{}t_\alpha{}t_\beta{}t_\alpha\underline{t_\delta{}t_\beta}t_\alpha\underline{t_\beta{}t_\delta}t_\alpha\\
&=&
t_\beta{}t_\alpha{}t_\gamma{}t_\alpha{}t_\beta{}t_\alpha{}t_\beta\underline{t_\delta{}t_\alpha{}t_\delta}t_\beta{}t_\alpha\\
&=&
t_\beta{}t_\alpha{}t_\gamma{}t_\alpha{}t_\beta{}t_\alpha{}t_\beta{}t_\alpha{}t_\delta\underline{t_\alpha{}t_\beta{}t_\alpha}\\
&=&
t_\beta{}t_\alpha{}t_\gamma{}t_\alpha{}t_\beta{}t_\alpha{}t_\beta{}t_\alpha\underline{t_\delta{}t_\beta}t_\alpha{}t_\beta\\
&=&
t_\beta{}t_\alpha{}t_\gamma{}t_\alpha{}t_\beta{}t_\alpha{}t_\beta{}t_\alpha{}t_\beta{}t_\delta{}t_\alpha{}t_\beta\\
&=&
t_\beta{}t_\alpha{}t_\gamma{}t_\alpha\underline{t_\beta\cdot{}t_\delta^{-1}}t_\delta\cdot{}t_\alpha{}t_\beta{}t_\alpha{}t_\beta{}t_\delta{}t_\alpha{}t_\beta\\
&=&
t_\beta{}t_\alpha{}t_\gamma{}t_\alpha{}t_\delta^{-1}t_\beta{}t_\delta{}t_\alpha{}t_\beta{}t_\alpha{}t_\beta{}t_\delta{}t_\alpha{}t_\beta\\
&=&
t_\beta{}t_\alpha{}t_\gamma\underline{t_\alpha{}t_\delta^{-1}\cdot{}t_\alpha^{-1}}t_\beta^{-1}\underline{t_\beta{}t_\alpha\cdot{}t_\beta}t_\delta{}t_\alpha{}t_\beta{}t_\alpha{}t_\beta{}t_\delta{}t_\alpha{}t_\beta\\
&=&
t_\beta{}t_\alpha{}t_\gamma{}t_\delta^{-1}t_\alpha^{-1}\underline{t_\delta{}t_\beta^{-1}}t_\alpha{}t_\beta\underline{t_\alpha{}t_\delta{}t_\alpha}t_\beta{}t_\alpha{}t_\beta{}t_\delta{}t_\alpha{}t_\beta\\
&=&
t_\beta{}t_\alpha{}t_\gamma{}t_\delta^{-1}t_\alpha^{-1}t_\beta^{-1}t_\delta{}t_\alpha{}t_\beta{}t_\delta{}t_\alpha\underline{t_\delta{}t_\beta}t_\alpha{}t_\beta{}t_\delta{}t_\alpha{}t_\beta\\
&=&
(t_\beta{}t_\alpha{}t_\gamma{}t_\alpha^{-1}t_\beta^{-1})(t_\beta{}t_\alpha{}t_\delta^{-1}t_\alpha^{-1}t_\beta^{-1})(t_\delta{}t_\alpha{}t_\beta)^4\\
&\overset{\mathrm{(3)(a)}}{=}&
t_{t_\beta{}t_\alpha(\gamma)}t_{t_\beta{}t_\alpha(\delta)}^{-1}(t_\delta{}t_\alpha{}t_\beta)^4\\
&=&
t_At_Bt_Ct_D\\
&\overset{(1)}{=}&
t_At_Bt_C\\
&\overset{(1)}{=}&
t_At_Bt_Ct_E\\
&=&
t_At_Bt_Ct_Et_\alpha{}t_\alpha^{-1}\\
&=&
t_{\delta(\mu,\alpha)},
\end{eqnarray*}
where in the underlines, we apply the relation~$(3)^\prime$.

Therefore the group presented in Theorem~\ref{main-thm-2} is isomorphic to the group presented in Theorem~\ref{main-thm-1}, and so $\M(N_{g,n})$.
Thus we finish the proof of Theorem~\ref{main-thm-2}.

\section{Proof of Theorems~\ref{main-thm-3} and \ref{main-thm-4}}\label{proof-main-thm-3-4}

We first note that Theorems~\ref{main-thm-3} and \ref{main-thm-4} can be shown by the arguments similar to Sections~\ref{proof-main-thm-1} and \ref{proof-main-thm-2} respectively.
However we prove these by giving isomorphisms from the groups presented in Theorems~\ref{main-thm-3} and \ref{main-thm-4} to the groups presented in Theorems~\ref{main-thm-1} and \ref{main-thm-2} respectively.

Let $H$ be the group presented in either Theorems~\ref{main-thm-3} or \ref{main-thm-4}.
When $g=1$, since $\U$ is the empty set, the presentation of $H$ is same to the presentation of the group presented in Theorems~\ref{main-thm-1} and \ref{main-thm-2}.
Hence $H$ is isomorphic to $\M(N_{1,n})$.
Therefore we suppose $g\geq2$.

Let $\eta:H\to\M(N_{g,n})$ be the natural homomorphism and $\nu:\M(N_{g,n})\to{H}$ the homomorphism defined as $\nu(t_{c;\theta})=t_{c;\theta}$ and $\nu(Y_{\mu,\alpha})=t_{\alpha;\theta}^{-1}U_{\mu,\alpha;\theta}$ for any $t_{c;\theta}\in\T$ and $Y_{\mu,\alpha}\in\Y$.
If $\nu$ is well-defined, $\nu$ is the inverse map of $\eta$ clearly.
So it suffices to show well-definedness of $\nu$, that is, we show that the correspondence $\nu(Y_{\mu,\alpha})=t_{\alpha;\theta}^{-1}U_{\mu,\alpha;\theta}$ does not depend on the choice of $\theta\in\{+_\alpha.-_\alpha\}$, and that any relation of $\M(N_{g,n})$ is satisfied in $H$.

First we show that $t_{\alpha;+_\alpha}^{-1}U_{\mu,\alpha;+_\alpha}=t_{\alpha;-_\alpha}^{-1}U_{\mu,\alpha;-_\alpha}$ in $H$.
By the relations~(2)~(a) and (3)~(a) of $H$, we calculate
\begin{eqnarray*}
U_{\mu,\alpha;\theta}t_{\alpha;\theta}U_{\mu,\alpha;\theta}^{-1}
&=&
t_{U_{\mu,\alpha;\theta}(\alpha);(U_{\mu,\alpha;\theta})_\ast(\theta)}\\
&=&
t_{\alpha;\theta^\prime}\\
&=&
t_{\alpha;\theta}^{-1},
\end{eqnarray*}
where $\theta^\prime$ is the inverse orientation of $\theta$, and so $t_{\alpha;\theta}^{-1}U_{\mu,\alpha;\theta}=U_{\mu,\alpha;\theta}t_{\alpha;\theta}$.
Since $t_{\alpha;-_\alpha}(\mu)=t_{\alpha;-_\alpha}(Y_{\mu,\alpha}(\mu))=U_{\mu,\alpha;-_\alpha}(\mu)$, $t_{\alpha;-_\alpha}(\alpha)=t_{\alpha;-_\alpha}(Y_{\mu,\alpha}(\alpha))=U_{\mu,\alpha;-_\alpha}(\alpha)$ and $(t_{\alpha;-_\alpha})_\ast(+_\alpha)=(t_{\alpha;-_\alpha})_\ast((Y_{\mu,\alpha})_\ast(-_\alpha))=(U_{\mu,\alpha;-_\alpha})_\ast(-_\alpha)$, by the relations~(2)~(a) and (3)~(b) of $H$, we calculate
\begin{eqnarray*}
t_{\alpha;-_\alpha}t_{\alpha;+_\alpha}^{-1}U_{\mu,\alpha;+_\alpha}
&=&
t_{\alpha;-_\alpha}U_{\mu,\alpha;+_\alpha}t_{\alpha;+_\alpha}\\
&=&
t_{\alpha;-_\alpha}U_{\mu,\alpha;+_\alpha}t_{\alpha;-_\alpha}^{-1}\\
&=&
U_{t_{\alpha;-_\alpha}(\mu),t_{\alpha;-_\alpha}(\alpha);(t_{\alpha;-_\alpha})_\ast(+_\alpha)}\\
&=&
U_{U_{\mu,\alpha;-_\alpha}(\mu),U_{\mu,\alpha;-_\alpha}(\alpha);(U_{\mu,\alpha;-_\alpha})_\ast(-_\alpha)}\\
&=&
U_{\mu,\alpha;-_\alpha}U_{\mu,\alpha;-_\alpha}U_{\mu,\alpha;-_\alpha}^{-1}\\
&=&
U_{\mu,\alpha;-_\alpha},
\end{eqnarray*}
and so $t_{\alpha;+_\alpha}^{-1}U_{\mu,\alpha;+_\alpha}=t_{\alpha;-_\alpha}^{-1}U_{\mu,\alpha;-_\alpha}$.
Therefore we conclude that the correspondence $\nu(Y_{\mu,\alpha})=t_{\alpha;\theta}^{-1}U_{\mu,\alpha;\theta}$ does not depend on the choice of $\theta\in\{+_\alpha.-_\alpha\}$.

Next we show that any relation appearing in Theorems~\ref{main-thm-1} and \ref{main-thm-2} is satisfied in $H$.
The relations~(1), (2)~(a), (3)~(a) with $f\in\T$, (4) and (5) in Theorems~\ref{main-thm-1} and \ref{main-thm-2} are satisfied in $H$ clearly.
By the relation~(2)~(b) of $H$ and the equality $t_{\alpha;\theta}^{-1}U_{\mu,\alpha;\theta}=U_{\mu,\alpha;\theta}t_{\alpha;\theta}$ shown above, we calculate
\begin{eqnarray*}
\nu(Y_{\mu,\alpha}^{-1})
&=&
(t_{\alpha;\theta}^{-1}U_{\mu,\alpha;\theta})^{-1}\\
&=&
(U_{\mu,\alpha;\theta}t_{\alpha;\theta})^{-1}\\
&=&
t_{\alpha;\theta}^{-1}U_{\mu,\alpha;\theta}^{-1}\\
&=&
t_{\alpha^{-1};\theta}^{-1}U_{\mu,\alpha^{-1};\theta}\\
&=&
\nu(Y_{\mu,\alpha^{-1}}).
\end{eqnarray*}
Hence the relation~(2)~(b) in Theorem~\ref{main-thm-2} is satisfied in $H$.
By the relations~(2)~(a) and (3)~(a) of $H$, we calculate
\begin{eqnarray*}
\nu(Y_{\mu,\alpha}t_{c;\theta}Y_{\mu,\alpha}^{-1})
&=&
(t_{\alpha;\theta}^{-1}U_{\mu,\alpha;\theta})t_{c;\theta}(t_{\alpha;\theta}^{-1}U_{\mu,\alpha;\theta})^{-1}\\
&=&
(t_{\alpha;\theta^\prime}U_{\mu,\alpha;\theta})t_{c;\theta}(t_{\alpha;\theta^\prime}U_{\mu,\alpha;\theta})^{-1}\\
&=&
t_{t_{\alpha;\theta^\prime}U_{\mu,\alpha;\theta}(c);(t_{\alpha;\theta^\prime}U_{\mu,\alpha;\theta})_\ast(\theta)}\\
&=&
t_{t_{\alpha;\theta}^{-1}U_{\mu,\alpha;\theta}(c);(t_{\alpha;\theta}^{-1}U_{\mu,\alpha;\theta})_\ast(\theta)}\\
&=&
t_{Y_{\mu,\alpha}(c);(Y_{\mu,\alpha})_\ast(\theta)}\\
&=&
\nu(t_{Y_{\mu,\alpha}(c);(Y_{\mu,\alpha})_\ast(\theta)}),
\end{eqnarray*}
where $\theta^\prime$ is the inverse orientation of $\theta$.
Hence the relations~(3)~(a) with $f\in\Y$ in Theorems~\ref{main-thm-1} and \ref{main-thm-2} are satisfied in $H$.
By the relations~(3)~(a) and (3)~(b) of $H$, we calculate
\begin{eqnarray*}
\nu(fY_{\mu,\alpha}f^{-1})
&=&
f(t_{\alpha;\theta}^{-1}U_{\mu,\alpha;\theta})f^{-1}\\
&=&
(ft_{\alpha;\theta}f^{-1})^{-1}(fU_{\mu,\alpha;\theta}f^{-1})\\
&=&
t_{f(\alpha);f_\ast(\theta)}^{-1}U_{f(\mu),f(\alpha);f_\ast(\theta)}\\
&=&
\nu(Y_{f(\mu),f(\alpha)}).
\end{eqnarray*}
Hence the relations~(3)~(b) in Theorems~\ref{main-thm-1} and \ref{main-thm-2} are satisfied in $H$.
By the relation~(6) of $H$ and the equality $t_{\alpha;\theta}^{-1}U_{\mu,\alpha;\theta}=U_{\mu,\alpha;\theta}t_{\alpha;\theta}$, we calculate
\begin{eqnarray*}
\nu(Y_{\mu,\alpha}^2)
&=&
(t_{\alpha;\theta}^{-1}U_{\mu,\alpha;\theta})^2\\
&=&
(U_{\mu,\alpha;\theta}t_{\alpha;\theta})(t_{\alpha;\theta}^{-1}U_{\mu,\alpha;\theta})\\
&=&
U_{\mu,\alpha;\theta}^2\\
&=&
t_{\delta(\mu,\alpha)}\\
&=&
\nu(t_{\delta(\mu,\alpha)}).
\end{eqnarray*}
Hence the relation~(6) in Theorem~\ref{main-thm-1} is satisfied in $H$.
Therefore we conclude that any relation of $\M(N_{g,n})$ is satisfied in $H$.

So we obtain well-definedness of $\nu$, and hence it follows that $H$ is isomorphic to $\M(N_{g,n})$.
Thus we complete the proof of Theorems~\ref{main-thm-3} and \ref{main-thm-4}.

\appendix
\section{Problems on presentations for $\M(N_{g,n})$}\label{problem}

In this appendix, we introduce three problems on presentations for $\M(N_{g,n})$ as follows.

In the main theorems, we can reduced relations.
For example, in the relation~(3)~(a) in Theorems~\ref{main-thm-1} and \ref{main-thm-2} (resp. Theorems~\ref{main-thm-3} and \ref{main-thm-4}), $\T\cup\Y$ (resp. $\T\cup\U$) can be reduced to a finite number of Dehn twists and one crosscap slide (resp. a finite number of Dehn twists and one crosscap transposition).
Similarly, in the relation~(3)~(b) in Theorems~\ref{main-thm-1} and \ref{main-thm-2} (resp. Theorems~\ref{main-thm-3} and \ref{main-thm-4}), $\T$ (resp. $\T\cup\{U_{\mu,\alpha;\theta}\}$) can be reduced to a finite number of Dehn twists (resp. a finite number of Dehn twists by adding the relation $U_{\mu,\alpha;+_\alpha}=U_{t_{\alpha;+_\alpha}(\mu),\alpha;-_\alpha}$).
In addition, we can reduced the relation~(6) in Theorem~\ref{main-thm-1} (resp. Theorem~\ref{main-thm-3}) to one relation $Y_{\mu_0,\alpha_0}^2=t_{\delta(\mu_0,\alpha_0)}$ (resp. $U_{\mu_0,\alpha_0;\theta}^2=t_{\delta(\mu_0,\alpha_0)}$) for some pair $(\mu_0,\alpha_0)$.
On the other hand, the relation of the mapping class group of an orientable surface corresponding to the relation~(3)~(a) with $f\in\T$ in the main theorems can be reduced to only all the commutativity relations and all the braid relations.
However, in the non-orientable case, we do not know whether or not  the same holds.
So we have a natural problem as follows.

\begin{prob}
Give an infinite presentation of $\M(N_{g,n})$ whose relations are more simple for $g\geq1$ and $n\geq0$.
\end{prob}

It is known that $\M(N_{g,n})$ can not be generated by only either Dehn twists or crosscap slides for $g\geq2$ and $n\geq0$ (see \cite{Li1, Li2}).
On the other hand, Le\'sniak-Szepietowski~\cite{LS} showed that $\M(N_{g,n})$ can be generated by only crosscap transpositions for $g\geq7$ and $n\geq0$.
So we have a natural problem as follows.

\begin{prob}
Give an infinite presentation of $\M(N_{g,n})$ with all crosscap transpositions for $g\geq7$ and $n\geq0$.
\end{prob}

It is known that $\M(N_{1,0})$ and $\M(N_{1,1})$ are trivial (see \cite{E}).
In addition, for $g\geq2$ and $n\leq1$, a simple finite presentation of $\M(N_{g,n})$ was given (see \cite{Li1, BC, St1, Sz, PS}).
Moreover, Omori and the author~\cite{KO1} gave a finite presentation of $\M(N_{g,n})$ for $g\geq1$ and $n\geq2$.
However this presentation is very complicated.
On the other hand, there is a simple finite presentation for the mapping class group of a compact orientable surface (see~\cite{G}).
So we have a natural problem as follows.

\begin{prob}
Give a simple finite presentation of $\M(N_{g,n})$ with Dehn twists and crosscap slides, or with Dehn twists and crosscap transpositions (resp. with crosscap transpositions) for $g\geq1$ (resp. $g\geq7$) and $n\geq0$.
\end{prob}


\end{document}